\documentclass[reqno]{amsart}
\usepackage{amsmath}
\usepackage{amssymb}
\usepackage{amsfonts}
\usepackage{url}
\usepackage{textcomp}
\usepackage[dvipdfmx]{graphicx}
\usepackage{comment}
\usepackage{setspace}

\makeatletter
\@namedef{subjclassname@2020}{%
  \textup{2020} Mathematics Subject Classification}
\makeatother

\newtheorem{theorem}{Theorem}
\newtheorem{corollary}[theorem]{Corollary}
\newtheorem{proposition}[theorem]{Proposition}
\newtheorem{lemma}[theorem]{Lemma}
\theoremstyle{definition}
\newtheorem{example}[theorem]{Example}
\usepackage{xcolor}

\newcommand{\cY}{{\mathcal Y}}
\newcommand{\cH}{{\mathcal H}}

\newcommand{\F}{\mathbb{F}}

\newcommand{\PG}{\mathrm{PG}}

\newcommand{\PGU}{\mathrm{PGU}}

\newcommand{\fqs}{\mathbb{F}_{q^6}}
\newcommand{\ord}{\mathrm{ord}}

\title{New sextics of genus 6 and 10 attaining the Serre bound}

\author[A. Iezzi]{Annamaria Iezzi}
\address{Dipartimento di Matematica e Applicazioni ``Renato Caccioppoli",
Università degli Studi di Napoli Federico II,
Via Cintia, Monte S. Angelo - I-80126 Napoli, Italy}
\email{annamaria.iezzi@unina.it}

\author[M. Q. Kawakita]{Motoko Qiu Kawakita}
\address{Division of Mathematics, Shiga University of Medical Science,
 Seta Tsukinowa-cho, Otsu, Shiga, 520-2192 Japan}
 \email{kawakita@belle.shiga-med.ac.jp}

\author[M. Timpanella]{Marco Timpanella}
\address{Dipartimento di Matematica e Informatica,
Università degli Studi di Perugia - University of Perugia
Via Vanvitelli, 1 - 06123 Perugia, Italy}
\email{marco.timpanella@unipg.it}


\begin{document}
\maketitle

\begin{abstract} 
We provide new examples of curves of genus 6 or 10 attaining the Serre bound. They all belong to the family of sextics introduced in \cite{k17} as a a generalization of the Wiman sextics \cite{w96} and Edge sextics \cite{e81}. Our approach is based on a theorem by Kani and Rosen which allows, under certain assumptions, to fully decompose the Jacobian of the curve. With our investigation we are able to update several entries in \url{http://www.manypoints.org} (\cite{ghlr}).

\end{abstract}

\section{Introduction}

Let $\mathbb{F}_{q}$ be a finite field with $q$ elements, where $q$ is a power of a prime $p$. A projective, absolutely irreducibile, non-singular, algebraic curve $C$ defined over $\F_q$ is called $\F_q$\emph{-maximal} if the number of its $\F_q$-rational points, denoted by $\#C(\F_q)$, attains the Hasse-Weil upper bound
$$
\#C(\F_q) \le q + 1 + 2g\sqrt{q},
$$
where $g$ is the genus of the curve. A classical and well-studied example of maximal curve is the so called \emph{Hermitian curve} $\mathcal{H}_q$ of affine equation $x^q+x=y^{q+1}$. This curve is $\mathbb F_{q^2}$-maximal, and it has the largest possible genus for an $\mathbb F_{q^2}$-maximal curve \cite{Ruck-Sti}. By a result commonly referred to as the Kleiman–Serre covering result \cite{Kleiman, Lachaud1}, any curve defined over $\F_q$ and $\F_q$-covered by an $\F_q$-maximal curve, is $\F_q$-maximal. This result provides a strong tool for constructing new examples of maximal curves from the known ones: indeed, if $C$ is an $\F_q$-maximal curve, also the quotient curve $C/G$, where $G$ is a finite subgroup of the automorphism group of $C$, is $\F_q$-maximal. Most of the known $\F_{q^2}$-maximal curves are obtained, following this approach, as quotient curves of the Hermitian curve $\mathcal{H}_q$, see for instance \cite{CKT,GSX,MZRS2} and the references therein. In the research community, for a while, it was even speculated  that  all maximal curves could be obtained as quotient curves of the Hermitian curve. However, in 2009, Giulietti and Korchm\'aros \cite{GK} proved that this
is false, as they constructed an $\F_{q^6}$-maximal curve not covered by the Hermitian curve when $q>2$. Since then, a few other examples of maximal curves not covered by the Hermitian curve have been provided \cite{BM,ggs10,TTT}, but a complete solution of the classification problem for maximal curves seems to be out of reach, and looking for new examples is a very active line of research.


Clearly, $\mathbb F_q$-maximal curves can only exist when $q$ is a  square. In 1983, Serre provided a non-trivial improvement of the Hasse-Weil bound when $q$ is not a square, namely
$$\#C(\F_q) \le q + 1 + g \lfloor 2\sqrt{q} \rfloor,$$ where $\lfloor\cdot\rfloor$ is the floor function (see \cite{Serre}). We refer to this bound as the \emph{Serre bound}.

Curves attaining the Hasse-Weil or the Serre bound  are interesting objects in their own right, but also for their applications in Coding Theory. Indeed, in \cite{Goppa}, Goppa described a way to use algebraic curves to construct linear error correcting
codes, the so called algebraic geometric codes (AG codes). As the relative Singleton defect of an AG code from a curve $C$ is upper bounded by the ratio $g/N$, where $g$ is the genus of $C$ and $N$ can be as large as the number of $\mathbb F_q$-rational points of $C$, it follows that curves with many rational points with respect to their genus are of great interest in Coding Theory. For this reason, maximal curves and curves attaining the Serre bound have been widely investigated in the last years, see for instance \cite{BMZ2, BMZ1, CT,GF,KNT,LV,LT,MTZ,TC}.

The aim of this paper is to provide new explicit examples of curves attaining the Hasse-Weil or the Serre bound.
We stress the fact that while there is a wide literature about maximal curves, there are only a few examples of curves attaining the Serre bound and which are not maximal. 

All the examples that we provide in this paper  have genus 6 or 10 and belong to the following generalization of the families of the Wiman sextics \cite{w96} and Edge sextics \cite{e81} introduced in \cite{k17}
\[x^6+y^6+1 +a (x^4y^2+x^2+y^4)
+b(x^2y^4+x^4+y^2)
+c x^2y^2=0,\]
with $a,b,c\in \F_q$. Our approach is based on a theorem by Kani and Rosen (see \cite[Theorem B]{kr89}), which provides a decomposition of the Jacobian of a curve
under certain conditions on its automorphism group. This theorem has already been used to find  examples of curves
with many rational points with respect to their genus, see for instance \cite{bgkm20,k15, k17, KLT}.

The paper is organized as follows. In Section \ref{W-E} we recall the defining equations of  the Wiman's and Edge's sextics and the statement of a theorem by Kani and Rosen. In Section \ref{tool} we prove a useful result to decompose completely the Jacobian of degree-6 hyperelliptic curve under certain assumptions (see Theorem \ref{JacHyper}). In Section \ref{sec:new_sextics} and \ref{wimanten} we use Theorem \ref{JacHyper} to decompose the Jacobians of certain Wiman's and Edge's sextics, and we find new explicit examples of curves attaining the Hasse-Weil bound or the Serre bound. Our investigation allows to update several entries in \url{http://www.manypoints.org}  (\cite{ghlr}). Finally, in Section \ref{Quotient}, we prove that one of the maximal curves that we exhibit in the previous section is not Galois-covered by the Hermitian curve.

\section{Families of Wiman's and Edge's sextics}\label{W-E}
In 1896 Wiman introduced in \cite{w96} the family of sextics defined by the equation
\[
W_{a,b}\colon x^6+y^6+1
+a (x^4y^2+x^2y^4+x^4+x^2+y^4+y^2)
+b x^2y^2=0,
\]
where $a,b\in\mathbb C$. Similarly, in 1981,  Edge considered in \cite{e81} the family of sextics described by the equation
\[E_\alpha \colon x^6+y^6+1+(x^2+y^2+1)(x^4+y^4+1)-12x^2y^2+\alpha(y^2-1)(1-x^2)(x^2-y^2)=0,\]
with $\alpha\in\mathbb C$.
The geometry of the Wiman--Edge pencil is studied in \cite{dfl18}. 

By considering Wiman's and Edge's sextics defined over a finite field rather than $\mathbb C$, in \cite{k15} and \cite{k11}  it has been shown that they provide several examples of $\mathbb F_{p^2}$-maximal curves, or of curves attaining the Serre bound over $\mathbb F_p$ and $\mathbb F_{p^3}$. This allowed to update several entries in \url{http://www.manypoints.org} (\cite{ghlr}).

By combining the defining equations of the Wiman's and Edge's  sextics, a generalization of these two families was introduced and studied in \cite{k17}:
\[S\colon x^6+y^6+1 +a (x^4y^2+x^2+y^4)
+b(x^2y^4+x^4+y^2)
+c x^2y^2=0.\]
It is readily seen that if $a=b$ then $S$ is a Wiman's sextic,
whereas if  $a=(1-\alpha)/2$, $b=(1+\alpha)/2$ and
$c=-6$ then $S$ is an Edge's sextic.

In this paper we focus on $S_{a,b}$, obtained from $S$ by setting $c=-3(a+b+1)$, (see Section \ref{sec:new_sextics}) and on $W_{a,b}$ (see Section \ref{wimanten}).

The idea for studying the number of $\mathbb F_q$-rational points of $S_{a,b}$ and $W_{a,b}$ is to decompose their Jacobian varieties. For this we will make  use of the following theorem by Kani and Rosen.
\begin{theorem}\cite[Theorem B]{kr89}
Let $C$ be a curve and let $G$ be a finite subgroup of the automorphism group $\operatorname{Aut}(C)$ such that $G=H_1\cup \cdots \cup H_n$, where the $H_i$'s are subgroups of $G$ with $H_i\cap H_j=\{1_G\}$ for $i\neq j$. Then the following isogeny relation holds:
$$\operatorname{Jac}(C)^{n-1}\times \operatorname{Jac}(C/G)^g\simeq \operatorname{Jac}(C/H_1)^{h_1}\times \cdots \times \operatorname{Jac}(C/H_n)^{h_n},$$
where $g=|G|$ and $h_i=|H_i|$.
\end{theorem}

As we will see in Sections \ref{sec:new_sextics} and \ref{wimanten} we can always decompose $\operatorname{Jac}(S_{a,b})$ and $\operatorname{Jac}(W_{a,b})$ as a product of elliptic curves and Jacobians of hyperelliptic curves of degree 6. This is why in the next section, in order to fully decompose the Jacobians as a product of elliptic curves, we first study the decomposition of the Jacobian variety of a degree-6 hyperelliptic curve.

\section{Decomposing the Jacobian of certain curves of genus two}\label{tool}

In this section, we deal with the decomposition of the Jacobian variety $\mathrm{Jac}(D)$ of
a hyperelliptic curve 
$$D\, \colon\, y^2=f(x),$$
where $f(x)$ is a degree 6 polynomial defined over $\F_q$. In particular we prove that, under certain assumptions on $f(x)$, the Jacobian $\mathrm{Jac}(D)$ splits completely as a product of two elliptic curves.
The idea of the proof comes from Section 1.2 of \cite{iko86},
where the authors work over an algebraically closed field, rather than a finite field.
\begin{theorem} \label{JacHyper} 
 Assume that the polynomial $f(x)$ factors
completely over the finite field $\F_q$, i.e.
$$\displaystyle f(x)=c \cdot \Pi_{i=1}^{6}(x-a_i)$$
with $a_i \in \F_q$ for $i=1, \ldots, 6$,
$a_i \neq a_j$ when $i \neq j$. Assume, moreover, that there exists a permutation of the roots such that
$$(a_2-a_4)(a_1-a_6)(a_3-a_5)=(a_2-a_6)(a_1-a_5)(a_3-a_4).$$
and set $$\lambda=\dfrac{(a_1-a_3)(a_2-a_4)}{(a_2-a_3)(a_1-a_4)}
\mu=\dfrac{(a_1-a_3)(a_2-a_5)}{(a_2-a_3)(a_1-a_5)},$$
 $$\theta=c\cdot(a_2-a_3)(a_1-a_4)(a_1-a_5)(a_1-a_6).$$
Then, we have the following:
\begin{enumerate}
\item[(a)] 
The hyperelliptic curve $D$ is isomorphic to
\begin{equation}\label{eq:hyperD}y^2=\theta x(x-1)(x-\lambda)(x-\mu)\left(x-\dfrac{\lambda(1-\mu)}{1-\lambda}\right)\end{equation}
over the finite field $\F_q$.
\item[(b)] Assume that  there exists a square root of $\lambda(\lambda-\mu)$ in the finite field $\F_q$. Then the Jacobian of the hyperelliptic curve $D$ decomposes over $\F_q$ as 
\[ \mathrm{Jac}(D) \sim E_\sigma \times E_\tau, \]
where $E_\sigma$ and $E_\tau$ are the elliptic curves defined  by the equations\textup{:}
$$ s^2=\dfrac{\theta(1-\mu)}{1-\lambda}t(t-1)
(t-\dfrac{(1-\lambda)
(\mu-2\lambda\pm2(\lambda^2-\lambda\mu)^{1/2})}
{\mu-1}).$$
Moreover, the number of rational points of $D$ satisfies
$$\#D(\F_q) =  \#E_\sigma(\F_q)+\#E_\tau(\F_q)-q-1.$$
\end{enumerate}
\end{theorem}
\begin{proof}
(a)
Let  us consider the  map defined by $$x \mapsto \dfrac{(a_1-a_3)(a_1-a_2)+x-a_1}{(x-a_1)(a_2-a_3)}, \quad y \mapsto \dfrac{(a_1-a_2)^2(a_1-a_3)^2y}{(a_2-a_3)^2(x-a_1)^3}.$$

With this map, $D$ is isomorphic to: 
$$y^2=\theta x(x-1)(x-\lambda)(x-\mu)(x-\nu),$$
with $\nu=\dfrac{(a_1-a_3)(a_2-a_6)}{(a_2-a_3)(a_1-a_6)}$.

Since $\nu=\dfrac{\lambda(1-\mu)}{1-\lambda}$ is equivalent to 
$$(a_2-a_4)(a_1-a_6)(a_3-a_5)=(a_2-a_6)(a_1-a_5)(a_3-a_4),$$
 we have (a).

(b)
Using Equation \eqref{eq:hyperD}, the following maps define three automorphisms of the hyperelliptic curve $D$:

$$\sigma \colon x \mapsto \dfrac{\lambda(x-\mu)}{x-\lambda},\quad 
y \mapsto y \dfrac{\lambda^{3/2}(\lambda-\mu)^{3/2}}{(x-\lambda)^3},$$
$$\iota \colon x \mapsto x,\quad y \mapsto -y,$$ and $\tau= \sigma \cdot \iota$. Moreover, as $\lambda(\lambda-\mu)$ has a square root in $\mathbb{F}_q$, these automorphisms are defined over $\mathbb{F}_q$.
Let $E_\sigma$ and $E_\tau$ be the quotient curves $D\slash \langle\sigma \rangle$ and $D\slash \langle\tau\rangle$, respectively.
By setting $t=x+\dfrac{\lambda(x-\mu)}{x-\lambda}$ 
and $s=y\dfrac{x-(\lambda\mp(\lambda^2-\lambda\mu)^{1/2})}{(x-\lambda)^2}$,
we have the following defining equations for $E_\sigma$ and $E_\tau$\textup{:}
$$ s^2=\theta(t-\mu)(t-\dfrac{1-\lambda\mu}{1-\lambda})
(t-2(\lambda\mp(\lambda^2-\lambda\mu)^{1/2})),$$
which are birationally equivalent to
$$ s^2=\dfrac{\theta(1-\mu)}{1-\lambda}t(t-1)
(t-\dfrac{(1-\lambda)
(\mu-2\lambda\pm2(\lambda^2-\lambda\mu)^{1/2})}
{\mu-1}).$$
Hence, 
we have that
	$\mathrm{Jac}(D) \sim E_\sigma \times E_\tau$
by Theorem B of \cite{kr89}. 

Let us now prove that $\#D(\F_q) =  \#E_\sigma(\F_q)+\#E_\tau(\F_q)-q-1$. 
It is well known that $\#D(\F_q)=q+1-t$, where $t$ is the trace of the Frobenius endomorphism
acting on a Tate module of $\mathrm{Jac}(D)$. Since $\mathrm{Jac}(D) \sim E_\sigma \times E_\tau$, then the Tate module of $D$ is
isomorphic to the direct sum of the Tate modules of $E_\sigma$ and $E_\tau$.
Hence $t=t_1+t_2$, where $t_1$ and $t_2$ are the traces of
the Frobenius on the Tate modules of $E_\sigma$ and $E_\tau$ respectively. The result follows by recalling that
$t_1=q+1-\#E_\sigma(F_q)$ and $t_2=q+1-\#E_\tau(\F_q)$.
\end{proof}

\section{New sextics of genus six attaining the Serre bound}\label{sec:new_sextics}

In this section  we consider the family of sextics
\[S_{a,b} \colon x^6+y^6+1 +a (x^4y^2+x^2+y^4)
+b(x^2y^4+x^4+y^2)
-3(a+b+1)x^2y^2=0.\]
We recall the following two results from \cite{k17}.
\begin{proposition} \cite[Proposition 1]{k17} \label{jac}
Let $k$ be a field of characteristic $p\geq 5$. Let $a,b\in k$ and let $D_{a,b} \colon y^2=f_{a,b}(x)$ be the hyperelliptic curve
 with
\[f_{a,b}(x)= -(x^3+bx^2+ax+1)((a+b+2)x^3-(a+2b+3)x^2+(b+3)x-1).\]
Then, the Jacobian variety of the sextic $S_{a,b}$ decomposes over  $k$ as
\[\mathrm{Jac}(S_{a,b}) \sim \mathrm{Jac}(D_{a,b}) \times \mathrm{Jac}(D_{b,a})^2.\] 
\end{proposition} 
Note that if  $a+b+2$,
$-4a^3 + a^2b^2 + 18ab - 4b^3 - 27$,
and the resultant of $x^3+bx^2+ax+1$ and
$(a+b+2)x^3-(a+2b+3)x^2+(b+3)x-1$
are not $0$
then the genus of the hyperelliptic curve $D_{a,b}$ is $2$, and therefore $S_{a,b}$ has genus 6. 

Moreover, since $\#D_{b,a}(\F_q) =\#D_{a,b}(\F_q)$, we have the following corollary.
\begin{corollary}\label{Cork17} \cite[Corollary 2]{k17}
 \label{coro_num}
When $a,b\in\mathbb F_q$, then the number of $\mathbb F_q$-rational points of the sextic $S_{a,b}$ is
 $$\#S_{a,b}(\F_q) = 3 \#D_{a,b}(\F_q)-2q-2.$$
\end{corollary}

By Corollary \ref{Cork17}, in order to compute $\#S_{a,b}(\F_q)$, it is enough to compute $\#D_{a,b}(\F_q)$. 

\begin{corollary}\label{Cor1}
When the polynomial $f_{a,b}(x)$ satisfies the assumptions in 
Theorem~\textup{\ref{JacHyper}},
then  the number of rational points of the sextic $S_{a,b}$ is
$$\#S_{a,b}(\F_q) = 3 \#E_\sigma(\F_q)+3\#E_\tau(\F_q)-5q-5,$$
where $E_\sigma$ and $E_\tau$ are defined as in Theorem~\textup{\ref{JacHyper}}(b).
\end{corollary}
\begin{proof}
The claim follows from Theorem~\textup{\ref{JacHyper}} and Corollary \ref{Cork17}.
\end{proof}

Because of Corollary \ref{Cor1}, we can determine the number of rational points of the sextic $S_{a,b}$ by computing the number of rational points of the two elliptic curves $E_{\sigma}$ and $E_{\tau}$.
Using this fact and a MAGMA(\cite{magma})-aided search, we were able to find new examples of curves of genus $6$ attaining the Serre bound. 

\begin{example} The sextic
$$S_{444, 469} \colon x^6+y^6+1+444(x^4y^2+x^2+y^4)+469(x^2y^4+x^4+y^2)+1239x^2y^2=0$$
has $1760$ rational points over the finite field $\F_{1327}$ and therefore it
 attains the Serre bound.
Its automorphism group has $12$ elements.

We can apply Theorem \ref{JacHyper}
to the hyperelliptic curve $D_{444, 469}\colon y^2=f_{444,469}(x)$, 
where $f_{444,469}(x)=c(x-a_1)(x-a_2)(x-a_3)(x-a_4)(x-a_5)(x-a_6)$,
with $c=412$, $a_1=548$, $a_2=541$, $a_3=289$, $a_4=364$, $a_5=344$, $a_6=28$.
Note that the order of $a_i$ for $i= 1, \ldots, 6$ is important.
With the notation of Theorem \ref{JacHyper}, we have $\lambda=611$, $\mu=656$ and $\theta=696$.
Hence, $ \mathrm{Jac}(D_{444,469}) \sim E_\sigma \times E_\tau$ where
$E_{\sigma} \colon s^2=247t(t-1)(t-811)$ and
$E_{\tau} \colon s^2=247t(t-1)(t-1084)$. So
$\#S_{444,469}(\F_q) = 3 \#E_\sigma(\F_q)+3\#E_\tau(\F_q)-5q-5.$
\end{example}

\begin{example} The sextic
$$S_{0, 7} \colon x^6+y^6+1+7(x^2y^4+x^4+y^2)+35x^2y^2=0$$
is maximal over the finite field $\F_{59^2}$. 
Its automorphism group has $12$ elements.

We can apply Theorem \ref{JacHyper}
to the hyperelliptic curve $D_{0, 7}\colon y^2=f_{0,7}(x)$ with 
$f_{0,7}(x)=50(x-39)(x-25)(x-8)(x-27)(x-23)(x-4)$.
Therefore, with the notation of Theorem \ref{JacHyper}, we have $\lambda=13$, $\mu=5$ and $\theta=33$.
Hence $ \mathrm{Jac}(D_{0,7}) \sim E_\sigma \times E_\tau$ with
$E_{\sigma} \colon s^2=11t(t-1)(t-30)$ and
$E_{\tau} \colon s^2=11t(t-1)(t-37)$.
So $\#S_{0,7}(\F_q) = 3 \#E_\sigma(\F_q)+3\#E_\tau(\F_q)-5q-5.$
\end{example}
Other examples for which $S_{a,b}$ is maximal over the finite field $\F_{p^2}$ are obtained for $(p,a,b)\in\{(59,0,7)$, $(71,13,23)$, $(79,9,18)$,
$(83,36,52)$, $(107,53,58)$, $(139,103,115)$, $(167,7,52)$,
$(179,131,154)$,
$(191,61,85)$,\,\ldots\}. We also observed that all the curves corresponding to the listed triplets have  automorphism group
of order $12$.

\begin{example} The sextic
$$S_{558, 2522} \colon x^6+y^6+1+558(x^4y^2+x^2+y^4)+2522(x^2y^4+x^4+y^2)+2425x^2y^2=0$$
attains the Serre bound over $\F_{2917^3}$.
Its automorphism group has $12$ elements.

We apply Theorem \ref{JacHyper}
to the hyperelliptic curve $D_{558, 2522}\colon y^2=f_{558,2522}(x)$ with 
$f_{558,2522}(x)=2752(x-2694)(x-2107)(x-2027)(x-1853)(x-2095)(x-879)$.
Therefore, $\lambda=171$, $\mu=932$ and $\theta=893$.
Hence, $ \mathrm{Jac}(D_{558,2522}) \sim E_\sigma \times E_\tau$ with
$E_{\sigma} \colon s^2=2677t(t-1)(t-1194)$ and
$E_{\tau} \colon s^2=2677t(t-1)(t-1495)$. 
So $\#S_{558,2522}(\F_q) = 3 \#E_\sigma(\F_q)+3\#E_\tau(\F_q)-5q-5.$
\end{example}
For $(p,a,b)\in\{(67,6,62)$, $(101,25,59)$,
$(673,40,460)$, $(677,1,76)$,
$(1153,65,957)$,
$(2113,287,438)$, $(2311,431,1253)$, $(2707,143,2372)$,
 $(2909,174,1715)$, $(2917,558,2522)$, $(3361,1062,1788),$ $\cdots\},$ 
the sextic $S_{a,b}$ attains the Serre bound over the finite field $\F_{p^3}$.
The automorphism group of the curves corresponding to  the above listed examples has order $12$, except when $(p,a,b)=(67,6,62)$, in which case it has order $60$.

\section{The Wiman sextics of genus ten attaining the Serre bound}
\label{wimanten}

We consider in this section the family of sextics 
$$W_{a,b}: x^6+y^6+1+a(x^4y^2+x^2y^4+x^4+x^2+y^4+y^2)+bx^2y^2=0.$$

In \cite{k15}, the case $b=-6a-3$ was investigated, and 
new examples of curves of genus $6$ attaining the Serre bound were provided.

In \cite{bgkm20}, the Jacobian of 
the sextic $W_{0,b}$ over the finite field $\F_{p^2}$ was completely decomposed, and this allowed to find new maximal curves of genus $10$. Moreover, in \cite{gkl20}, the Jacobian of $W_{0,b}$ was completely decomposed over any field $k$ whose characteristic $> 5$, bringing to find a new sextic of genus $10$ attaining the Serre bound.

With a computer search on the family of genus $10$ sextics  $W_{a,b}$, we found new examples of curves with genus 10 and many rational points. 
We list such examples
in Table~\ref{table:genus10} and we highlight how they improve the old entries from \url{http://www.manypoints.org} (\cite{ghlr}).

\begin{table}[h]
\caption{Wiman sextics of genus $10$ with many points}\label{table:genus10}
\begin{center}
\begin{tabular}{cccccc}
\hline
$p$& $a$&$b$& $\#W_{a,b}(\F_p)$ &old entry  &new entry\\
\hline
19&$0$& $0$& $72$ & $-100$ & $72-100$\\
23&$4$& $8$& $84$ & $-114$ & $84-114$\\
29&$8$& $1$& $102$ & $-130$& $102-130$\\
31&$0$& $12$& $108$ & $-133$& $108-133$\\
43&$36$& $28$& $132$ & $-170$& $132-170$\\
47&$19$& $0$& $144$ & $-178$& $144-178$\\
53&$42$& $46$& $162$ & $-191$& $162-191$\\
59&$5$& $21$& $168$ & $-205$& $168-205$\\
67&$17$& $17$& $192$ & $-222$& $192-222$\\
79&$0$& $30$& $216$ & $-247$& $216-247$\\
83&$13$& $79$& $216$ & $-256$& $216-256$\\
89&$37$& $43$& $246$ & $-267$& $246-267$\\
97&$48$&$5$&$246$&$-284$& $246-284$\\
\hline
\end{tabular}
\end{center}
\label{many}
\end{table}


In \cite{k15}, we also proved the following result on the decomposition of the Jacobian of $W_{a,b}$ over a field $k$ whose characteristic $> 5$. 

\begin{proposition}\label{Propk15} \cite[Proposition 10]{k15} \label{JacWab} The Jacobian of the Wiman sextic $W_{a,b}$ over
a field $k$ satisfies the following isogeny relation\textup{:}
$$\mathrm{Jac}(W_{a,b}) \sim V_1^3 \times V_2 \times \mathrm{Jac}(V_3)^3,$$
where $V_1, V_2$ and $V_3$ are defined by
\begin{align*}
V_1\colon&y^2=((3a-b-3)x-a+3)(1+(a-3)x(1-x)),\\
V_2\colon&x^3+y^3+1+a(x^2y+xy^2+x^2+x+y^2+y)+bxy=0,\\
V_3\colon&y^2=-((a+1)x^3+(2a+b)x^2+4ax+4)(x^3+ax^2+ax+1).
\end{align*}
\end{proposition}

In the cases where Theorem \ref{JacHyper} applies to the hyperelliptic curve $V_3$ of genus $2$, we can completely decompose the Jacobian of $V_3$, and hence we can completely decompose the Jacobian of $W_{a,b}$.

With a computer search we were able to find the following examples.


\begin{example}
The sextic $$W_{5,17}:x^6+y^6+1+5(x^4y^2+x^2y^4+x^4+x^2+y^4+y^2)+17x^2y^2=0$$
has $990$ rational points over $\F_{23^2}$ and it is a maximal curve of genus $10$.
Its automorphism group has $24$ elements. In  Section \ref{Quotient}, we will prove that this maximal curve is not a quotient of the Hermitian curve $\mathcal{H}_{23}$ defined over $\mathbb{F}_{23^2}$.

From Propostion \ref{JacWab}, we have that 
$\mathrm{Jac}(W_{5,17}) \sim V_1^3 \times V_2 \times \mathrm{Jac}(V_3)^3.$
By applying Theorem \ref{JacHyper} to
 $V_3\colon y^2=17(x-22)(x-14)(x-13)(x-11)(x-6)(x-5),$
we obtain that 
$\mathrm{Jac}(V_3) \sim E_1 \times E_2$
where
$$E_1:s^2=20t(t-1)(t-21),\quad
E_2:s^2=20t(t-1)(t-22).$$
Since $V_1$, $V_2$ are birational equivalent  to
$$E_3:s^2=t(t-1)(t-3),\quad
E_4:s^2=6t(t-1)(t-22).$$
respectively, we have that 
$$\mathrm{Jac}(W_{5,17}) \sim E_1^3 \times E_2^3 \times E_3^3 \times E_4.$$

\end{example}

For $(p,a,b)\in \{(167,27,40)$, $(191,49,131)$, $(239,119,216)$,$(263,51,123)$,
$(431,257,322)$, $(503,33,274)$, $(599,352,358)$, $(719,254,557)$, $(887,388,609)$,
$\ldots$\}, the Wiman sextic $W_{a,b}$ is maximal over the finite field $\F_{p^2}$.
The automorphism group of all the explicit examples listed here
has $24$ elements.
\begin{example}
The sextic $$W_{7,120}:x^6+y^6+1+7(x^4y^2+x^2y^4+x^4+x^2+y^4+y^2)+120x^2y^2=0$$
has $7242678$ rational points over $\F_{193^3}$.
It has genus $10$ and it attains the Serre bound.
Its automorphism group has $24$ elements.

From Proposition \ref{JacWab}, we have that 
$\mathrm{Jac}(W_{7,120}) \sim V_1^3 \times V_2 \times \mathrm{Jac}(V_3)^3.$
Applying Theorem \ref{JacHyper} to
 $V_3\colon y^2=185(x-192)(x-122)(x-101)(x-110)(x-89)(x-86),$
 we have that  
$\mathrm{Jac}(V_3) \sim E_1 \times E_2$
where
$$E_1:s^2=38t(t-1)(t-56), \quad
E_2:s^2=38t(t-1)(t-7).$$
Since $V_1$, $V_2$ are birational equivalent to
$$E_3:s^2=14t(t-1)(t-164), \quad
E_4:s^2=88t(t-1)(t-173)$$
respectively, we have that 
$$\mathrm{Jac}(W_{7,120}) \sim E_1^3 \times E_2^3 \times E_3^3 \times E_4.$$
\end{example}

For $(p,a,b)= (2909,2271,2350)$, $(4349,1169,4282)$ 
the sextic $W_{a,b}$ attains the Serre bound over the finite field $\F_{p^3}$.
Both automorphism groups have $24$ elements.
\section{$W_{5,17}$ is not Galois-covered by $\mathcal{H}_{23}$ over $\mathbb{F}_{23^2}$}\label{Quotient}

In the previous section we proved that the sextic $W_{5,17}$ is an $\F_{23^2}$-maximal curve of genus $10$. In this section, we prove that this maximal curve is not a quotient curve of the Hermitian curve $\mathcal{H}_{23}$ defined over $\mathbb{F}_{23^2}$.


Recall that every subgroup $G$ of the automorphism group of $\mathcal{H}_q$, which is isomorphic to $\PGU(3,q)$, produces a quotient curve $\cH_q/G$, and the cover $\cH_q\rightarrow\cH_q/G$ is a Galois cover defined over $\mathbb{F}_{q^2}$. The degree of the different divisor $\Delta$ of this covering is given by the Riemann-Hurwitz formula \cite[Theorem 3.4.13]{Sti},
\begin{equation} \label{RHformula}
\Delta=(2g(\cH_q)-2)-|G|(2g(\cH_q/G)-2).
\end{equation}

On the other hand, $$\Delta=\sum_{\sigma\in G\setminus\{id\}}i(\sigma),$$ where $i(\sigma)\geq 0$ is given by the Hilbert's different formula \cite[Theorem 3.8.7]{Sti}, namely
\begin{equation}\label{contributo}
i(\sigma)=\sum_{P\in\cH_q(\bar{\mathbb{F}}_q)}v_P(\sigma(t)-t),
\end{equation}
where $t$ is a local parameter at $P$.

By analyzing the geometric properties of the elements $\sigma \in \PGU(3,q)$, it turns out that there are only a few possibilities for $i(\sigma)$. This is stated in the next results on how an element of a given order in  $\PGU(3,q)$ acts on the set of $\mathbb{F}_{q^2}$-rational points of $\cH_q$, denoted by $\cH_q(\mathbb{F}_{q^2})$. We recall that a linear collineation $\sigma$ of $\PG(2,\mathbb{K})$ is a $(P,\ell)$-\emph{perspectivity}, if $\sigma$ preserves  each line through the point $P$ (the \emph{center} of $\sigma$), and fixes each point on the line $\ell$ (the \emph{axis} of $\sigma$). A $(P,\ell)$-perspectivity is either an \emph{elation} or a \emph{homology} according as $P\in \ell$ or $P\notin\ell$. 

\begin{lemma}{\rm{(}\cite[Lemma 2.2]{MZRS}\rm{)}}\label{classificazione}
For a nontrivial element $\sigma\in \PGU(3,q)$, one of the following cases holds.
\begin{itemize}
\item[(A)] ${\rm ord}(\sigma)\mid(q+1)$ and $\sigma$ is a homology whose center $P$ is a point of $\cH_q$ and whose axis $\ell$ is a chord of $\cH_q(\mathbb{F}_{q^2})$ such that $(P,\ell)$ is a pole-polar pair with respect to the unitary polarity associated to $\cH_q(\mathbb{F}_{q^2})$.
\item[(B)] ${\rm ord}(\sigma)$ is coprime to $p$ and $\sigma$ fixes the vertices $P_1,P_2,P_3$ of a non-degenerate triangle $T$.
\begin{itemize}
\item[(B1)] The points $P_1,P_2,P_3$ are $\fqs$-rational, $P_1,P_2,P_3\notin\cH_q$ and the triangle $T$ is self-polar with respect to the unitary polarity associated to $\cH_q(\mathbb{F}_{q^2})$. Also, $\ord(\sigma)\mid(q+1)$.
\item[(B2)] The points $P_1,P_2,P_3$ are $\fqs$-rational, $P_1\notin\cH_q$, $P_2,P_3\in\cH_q$. 
     Also, $\ord(\sigma)\mid(q^2-1)$ and $\ord(\sigma)\nmid(q+1)$.
\item[(B3)] The points $P_1,P_2,P_3$ have coordinates in $\mathbb{F}_{q^6}\setminus\mathbb{F}_{q^2}$, $P_1,P_2,P_3\in\cH_q$. 
    Also, $\ord(\sigma)\mid (q^2-q+1)$.
\end{itemize}
\item[(C)] ${\rm ord}(\sigma)=p$ and $\sigma$ is an elation whose center $P$ is a point of $\cH_q$ and whose axis $\ell$ is a tangent of $\cH_q(\mathbb{F}_{q^2})$; here $(P,\ell)$ is a pole-polar pair with respect to the unitary polarity associated to $\cH_q(\mathbb{F}_{q^2})$.
\item[(D)] ${\rm ord}(\sigma)=p$ with $p\ne2$, or ${\rm ord}(\sigma)=4$ and $p=2$. In this case $\sigma$ fixes an $\fqs$-rational point $P$, with $P \in \cH_q$, and a line $\ell$ which is a tangent  of $\cH_q(\mathbb{F}_{q^2})$; here $(P,\ell)$ is a pole-polar pair with respect to the unitary polarity associated to $\cH_q(\mathbb{F}_{q^2})$.
\item[(E)] $p\mid{\rm ord}(\sigma)$, $p^2\nmid{\rm ord}(\sigma)$, and ${\rm ord}(\sigma)\ne p$. In this case $\sigma$ fixes two $\fqs$-rational points $P,Q$, 
     with $P\in\cH_q$, $Q\notin\cH_q$. 
\end{itemize}
\end{lemma}

In the rest of the section, a nontrivial element of $\PGU(3,q)$ is said to be of type (A), (B), (B1), (B2), (B3), (C), (D), or (E), as given in Lemma \ref{classificazione}.

\begin{lemma}{\rm{(}\cite[Theorem 2.7]{MZRS}\rm{)}}\label{Lem:copertura}
For a nontrivial element $\sigma\in \PGU(3,q)$ one of the following cases occurs.
\begin{enumerate}
\item If $\ord(\sigma)=2$ and $2\mid(q+1)$, then $i(\sigma)=q+1$.
\item If $\ord(\sigma)=3$, $3 \mid(q+1)$ and $\sigma$ is of type {\rm(B3)}, then $i(\sigma)=3$.
\item If $\ord(\sigma)\ne 2$, $\ord(\sigma)\mid(q+1)$ and $\sigma$ is of type {\rm(A)}, then $i(\sigma)=q+1$.
\item If $\ord(\sigma)\ne 2$, $\ord(\sigma)\mid(q+1)$ and $\sigma$ is of type {\rm(B1)}, then $i(\sigma)=0$.
\item If $\ord(\sigma)\mid(q^2-1)$ and $\ord(\sigma)\nmid(q+1)$, then $\sigma$ is of type {\rm(B2)} and $i(\sigma)=2$.
\item If $\ord(\sigma)\ne3$ and $\ord(\sigma)\mid(q^2-q+1)$, then $\sigma$ is of type {\rm(B3)} and $i(\sigma)=3$.
\item If $p=2$ and $\ord(\sigma)=4$, then $\sigma$ is of type {\rm(D)} and $i(\sigma)=2$.
\item If $\ord(\sigma)=p$, $p \ne2$ and $\sigma$ is of type {\rm(D)}, then $i(\sigma)=2$.
\item If $\ord(\sigma)=p$ and $\sigma$ is of type {\rm(C)}, then $i(\sigma)=q+2$.
\item If $\ord(\sigma)\ne p$, $p\mid\ord(\sigma)$ and $\ord(\sigma)\ne4$, then $\sigma$ is of type {\rm(E)} and $i(\sigma)=1$.
\end{enumerate}
\end{lemma}

Finally, the following useful corollary can be deduced from the proof of Theorem 5 in \cite{GK}.
\begin{proposition}\label{boundgruppo}
Let $\cY$ be the quotient curve $\cH_q/G$, where $G$ is a subgroup of $PGU(3,q)$. Then 
$$
\frac{|\cH_{q}(\mathbb{F}_{q^2})|}{|\cY(\mathbb{F}_{q^2}
)|}\leq |G|\leq \frac{2g(\cH_{q})-2}{2g(\cY)-2}.
$$
\end{proposition}

\begin{proposition}
$W_{5,17}$ is not Galois-covered by $\mathcal{H}_{23}$ over $\mathbb{F}_{23^2}$.   
\end{proposition}
\begin{proof}
Assume, by the way of contradiction, that $W_{5,17}$ is Galois-covered by $\mathcal{H}_{23}$ over $\mathbb{F}_{23^2}$. So, there exists $G\leq \text{PGU}(3,23)$ such that $W_{5,17}$ is the quotient curve $\mathcal{H}_{23}/G$. First, by Proposition \ref{boundgruppo}, it follows that
$$
12<\frac{|\cH_{23}(\mathbb{F}_{23^2})|}{|W_{5,17}(\mathbb{F}_{23^2}
)|}\leq |G|\leq \frac{2g(\cH_{23})-2}{2g(W_{5,17})-2}=28.
$$
Furthermore, as $|G|$ must divide $|\text{PGU}(3,23)|=2^7\cdot 3^3\cdot 11 \cdot 13^2\cdot 23^3$, we have that $|G|\in \{13, 16,18,22,23,24,26,27\}$. Observe that, by the Riemann-Hurwitz formula, the degree of the different divisor $\Delta$ can be computed as 
\begin{equation}\label{eq:diff}
\Delta=(2g(\cH_{23})-2)-|G|(2g(W_{5,17})-2).
\end{equation}
To conclude the proof, we will rule out each of these possibilities for $|G|$.
\begin{itemize}
\item $|G|=13$. In this case $G$ contains exactly $12$ elements of order $13$, and hence Lemma \ref{Lem:copertura} yields $\Delta=12\cdot 3=36$, a contradiction with Equation \eqref{eq:diff}.
\item \textbf{$|G|=16$}. A MAGMA aided computation shows that there are $9$ subgroups (up to conjugation) of order $16$ in $\PGU(3,23)$, namely \textit{$G_i=S[i]$} where \textit{$S:=Subgroups(\PGU(3,23): OrderEqual:=16)$} and $i\in\{1,\ldots,9\}$. Let $N_i$ be the normalizer of $G_i$ in $\PGU(3,23)$ and $Q_i$ be the factor group $N_i/G_i$. Then, from Galois theory, $Q_i$ is a subgroup of $Aut(W_{5,7}) \cong {\rm S}_4$. By direct check with MAGMA, the order of $Q_i$ is not a divisor of $24$ for $i\in \{1,\ldots,5\}$, a contradiction. In the remaining cases we have $|Q_6|=|Q_7|=24$ and $|Q_8|=|Q_9|=12$.  However, all such groups are abelian, a contradiction with the structure of ${\rm S}_4$.

\item \textbf{$|G|=18$}. By Lemma \ref{Lem:copertura}, $G$ does not contain any non-trivial elements of orders $9$ and $18$. Therefore, by Sylow theorems, it contains exactly $8$ elements of order $3$, and Lemma \ref{Lem:copertura} yields
$$
180=\Delta=24\cdot i_2+8\cdot k+(17-8-i_2)\cdot l,
$$
where $i_2\in \{1,3,9\}$ is the number of involutions in $G$, $k\in \{0,3,24\}$ and $l\in\{0,24\}$. This is a contradiction as $180$ is not divisible by $8$.

\item \textbf{$|G|=22$}. In this case $G$ is isomorphic to either the cyclic or the dihedral group of order $22$. Now, Lemma \ref{Lem:copertura} yields
$$
\Delta=24+10\cdot 2+10\cdot 2,
$$
in the former case, and
$$
\Delta=24\cdot 11+10\cdot 2,
$$
in the latter case. In both cases this is a contradiction with $\Delta=108$.
\item \textbf{$|G|=23$}. In this case $G$ contains $22$ elements of order $23$, whence Lemma \ref{Lem:copertura} yields
$$
\Delta=22\cdot i,
$$
where $i\in\{2,25\}$. This is a contradiction with $\Delta=90$.

\item \textbf{$|G|=24$}. First observe that if $G$ contains more than $3$ involutions, Lemma \ref{Lem:copertura} yields $\Delta>24\cdot 3$, a contradiction with $\Delta=72$. Checking with MAGMA, there are $20$ subgroups (up to conjugation) of order $24$ in $\PGU(3,23)$ with at most $3$ involutions, namely $G_i=S[i]$ where \textit{$S:=Subgroups(\PGU(3,23): OrderEqual:=24);$} and $i\in\{1,\ldots,18\}\cup\{22,23\}$. 
For $i\in\{1,\ldots,18\}\cup\{22,23\}$, let $N_i$ be the normalizer of $G_i$ in $\PGU(3,23)$ and $Q_i$ be the factor group $N_i/G_i$. Then, if $i\in\{1,\ldots,10\}$, the order of $Q_i$ is not a divisor of $24$, a contradiction with $Q_i$ being a subgroup of ${\rm S}_4$. On the other hand, in the remaining cases $Q_i$ contains an abelian subgroup of order $12$, a contradiction with the structure of ${\rm S}_4$.
 
\item \textbf{$|G|=26$}. By Sylow theorems $G$ contains exactly $12$ elements of order $13$. Therefore Lemma \ref{Lem:copertura} yields
$$
\Delta\geq 12\cdot 3+24\cdot i,
$$
where $i\geq 1$ is the number of involutions in $G$. This is a contradiction with $\Delta=36$.

\item \textbf{$|G|=27$}. Checking with MAGMA, there is a unique (up to conjugation) subgroup of order $27$ in $\PGU(3,23)$, which is generated by the following automorphisms of $\mathcal{H}_{23}$:
$$
\left\{[x:y:z]\rightarrow [a_1x:a_2y:a_3z]\,\mid\, a_i^3=1, i=1,2,3
\right\}
$$
and
$$
[x:y:z]\rightarrow [z:x:y].
$$
A MAGMA computation shows that the quotient curve of $\cH_{23}$ by this group has genus $7$, a contradiction with $W_{5,17}$ having genus $10$.\qedhere
\end{itemize}
\end{proof}

\section{Acknowledgements} 

This research was supported by the Italian National Group for Algebraic and Geometric Structures and their Applications (GNSAGA - INdAM). The second author is funded by JSPS Grant-in-Aid for Scientific Research (C) 17K05344. The third author is funded by the project ``Metodi matematici per la firma digitale ed il cloud computing" (Programma Operativo Nazionale (PON) ``Ricerca e Innovazione" 2014-2020, University of Perugia).

\bibliographystyle{abbrv}
\bibliography{bibliography.bib}
\end{document}